\documentclass{amsart}
\usepackage{amssymb}
\usepackage{amsfonts}
\usepackage{amscd}
\usepackage{amsmath}
\usepackage{mathtools}
\usepackage[shortlabels]{enumitem}
\usepackage{epsfig}
\usepackage{graphicx}
\usepackage{overpic}
\usepackage{verbatim}
\usepackage{latexsym}
\usepackage{eucal}
\usepackage{bbm}

\usepackage[xcolor]{mdframed}

\usepackage{mathtools, array, hyperref}

\newtheorem{theorem}{Theorem}[section]
\theoremstyle{plain}

\newtheorem{lemma}[theorem]{Lemma}
\newtheorem{proposition}[theorem]{Proposition}
\theoremstyle{remark}

\newtheorem{example}[theorem]{Example}
\newmdtheoremenv{exer}{Exercise}[section]
\numberwithin{equation}{section}

\newcommand{\tr}{\operatorname{tr}}

\newcommand{\diam}{\operatorname{diam}}
\newcommand{\dist}{\operatorname{dist}}

\newcommand{\bbR}{\mathbb{R}}

\newcommand{\bbZ}{\mathbb{Z}}

\newcommand{\bbN}{\mathbb{N}}
\newcommand{\calA}{\mathcal{A}}

\newcommand{\calV}{\mathcal{V}}

\newcommand{\calJ}{\mathcal{J}}

\newcommand{\del}{\partial}

\DeclarePairedDelimiter\paren{(}{)}
\DeclarePairedDelimiter\set{\{}{\}}
\DeclarePairedDelimiter\sqbrak{[}{]}
\DeclarePairedDelimiter\abs{|}{|}
\DeclarePairedDelimiter\norm{\Vert}{\Vert}
\DeclarePairedDelimiter\brak{\langle}{\rangle}

\newcommand{\std}{{\rm N}}
\newcommand{\girth}{\operatorname{girth}}

\begin{document}

\title[Gaps between consecutive eigenvalues]{Gaps between consecutive eigenvalues for compact metric graphs}
\author[Borthwick]{David Borthwick}
\address{Department of Mathematics, Emory University, Atlanta, GA 30322}
\email{dborthw@emory.edu}
\author[Harrell]{Evans M. Harrell II}
\address{School of Mathematics, Georgia Institute of Technology, Atlanta, GA 30332}
\email{harrell@math.gatech.edu}
\author[Yu]{Haozhe Yu}
\address{Department of Mathematics, Emory University, Atlanta, GA 30322}
\email{haozhe.yu@emory.edu}
\date{\today}

\subjclass[2010]{34B45, 81Q35}

\begin{abstract}
On a compact metric graph, we consider the spectrum of the Laplacian defined with a mix of standard and Dirichlet vertex conditions. 
A Cheeger-type lower bound on the gap $\lambda_2 - \lambda_1$ is established, with a constant that depends only on
the total length of the graph and minimum edge length.
We also prove some improvements of known upper bounds for eigenvalue gaps and ratios for metric trees and extensions to certain other types of graphs.
\end{abstract}

\maketitle

\section{Introduction}

Estimating the gaps between eigenvalues is a classical problem in spectral theory, dating back to the upper bounds established 
for planar domains by Payne, P\'olya, and Weinberger \cite{PPW:1956}. For the Dirichlet problem on a bounded domain 
$\Omega \subset \bbR^d$, their
``universal bound'' on eigenvalue gaps generalizes to 
\[
\lambda_{n+1} - \lambda_n \le \frac{4}{nd} \sum_{j=1}^n \lambda_j.
\] 
For $\Omega$ convex, a lower bound $\lambda_2-\lambda_1 \ge 4\pi^2/\diam(\Omega)^2$ 
was proven using gradient methods in Singer-Wong-Yau-Yau \cite{SWYY:1985}. For the Dirichlet problem on a 
compact manifold with boundary, $\lambda_2-\lambda_1$ was bounded below in terms of a weighted 
Cheeger constant in Cheng-Oden \cite{ChengOden:1997} and Oden-Sung-Wang \cite{OSW:1999}.

In this paper we study the corresponding problems for the Laplacian on a connected, compact metric graph $\Gamma$, with 
either standard (Kirchhoff-Neumann) or Dirichlet conditions imposed at each vertex. More specifically, the vertex set $\calV$ of $\Gamma$ is
subdivided as $\calV_\std \cup \calV_0$, and $-\Delta$ is the self-adjoint operator associated to the quadratic form 
$\norm{u'}^2$ with domain $\set{u \in H^1(\Gamma): u = 0 \text{ on }\calV_0}$. This definition assigns standard vertex conditions
to the elements of $\calV_\std$.  Because Dirichlet conditions have the effect
of separating edges, all vertices in $\calV_{0}$ are assumed to have degree $1$. Vertices of degree $2$ are effectively
invisible under standard vertex conditions, so for $\calV_\std$ we assume the degree is either $1$ or $\ge 3$.

The spectrum of $-\Delta$ is given by
\[
0 \le \lambda_1 < \lambda_2 \le \lambda_3 \le \dots.
\]
Our main concern is the spectral gap $\lambda_2 - \lambda_1$, but first let us review some of the
existing eigenvalue estimates in this context.

In the case of fully standard vertex conditions, where $\calV_0 = \emptyset$, a variety of interesting eigenvalue 
estimates are known.  Following the notation of \cite{BKKM:2019, BKKM:2022}, we denote the eigenvalues 
under standard vertex conditions by $\lambda^\std_j$.
In terms of the total length $L := \abs{\Gamma}$, the standard eigenvalues satisfy
\begin{equation}\label{std.lower}
\lambda^\std_k \ge \frac{k^2 \pi^2}{4L^2}
\end{equation}
for all $k \ge 2$.  This lower bound was proven for $k=2$ by Nicaise \cite[Thm.~3.1]{Nicaise:1987}, for $k \ge 2$ by
Friedlander \cite[Thm.~1]{Friedlander:2005}, and independently by Kurasov-Naboko \cite{KN:2014} for $k$ even. 
The bound is sharp, and equality for some $k$ implies that $\Gamma$ is a segment if $k=2$ and an equilateral $k$-star if $k > 2$.

Nicaise also established a Cheeger-type lower bound,
\begin{equation}\label{std.cheeger}
\lambda^\std_2 \ge \frac14 h(\Gamma)^2,
\end{equation}
where the metric graph version of the the Cheeger constant is defined  by
\begin{equation}\label{cheeger.def}
h(\Gamma) := \inf \frac{\#S}{\min(\abs{Y_1},\abs{Y_2})},
\end{equation}
with the infimum taken over finite sets $S$ such that $\Gamma\backslash S$ is a disjoint union 
of non-empty open sets $Y_1$ and $Y_2$.  The Cheeger constant satisfies a trivial lower bound
\begin{equation}\label{h.lbound}
h(\Gamma) \ge \frac2{L},
\end{equation}
but this gives a much weaker result than \eqref{std.lower}.
If $D := \diam(\Gamma)$ satisfies $D \le L/2$, then Kennedy et al.~\cite[Thm.~7.2]{KKMM:2016} showed that 
\[
\lambda^\std_2 \ge \frac{1}{2DL}.
\]
This improves on \eqref{std.lower} for graphs of sufficiently small diameter.

Upper bounds for $\lambda^\std_2$ can be proven by a combination of min-max and surgery principles.
For example, if the number of edges of $\Gamma$ is $E \ge 2$, then
\cite[Thm.~4.2]{KKMM:2016} gives the sharp upper bound
\[
\lambda^\std_2 \le \frac{\pi^2E^2}{L^2}.
\]
Estimates based on other geometric properties of $\Gamma$ can be found, for example, in
\cite{BKKM:2017, BKKM:2019, BKKM:2022, BCJ:2021, KKMM:2016, KMN:2013}.

Returning to the general case, if $\calV_0$ is not empty then the arguments leading to \eqref{std.lower}
imply a lower bound
\begin{equation}\label{lambda1.lower}
\lambda_k \ge \frac{k^2 \pi^2}{4L^2}.
\end{equation}
This was proven for $k=1$ in \cite[Thm.~3.1]{Nicaise:1987} and is implicit in the proof of \cite[Thm.~1]{Friedlander:2005} for $k \ge 1$.
Berkolaiko-Kennedy-Kurasov-Mugnolo \cite[Thm.~4.7]{BKKM:2017} showed that, if $\Gamma$ is not a cycle
the bound can be improved for $k \ge E - V_0 + 1$ to
\[
\lambda_k \ge \paren[\Big]{k - \tfrac12(E - V_0 + 1)}^{\!2} \frac{\pi^2}{L^2},
\]
where $V_0 := \#\calV_0$.
The argument for the Cheeger estimate \eqref{std.cheeger} from \cite[Thm.~3.2]{Nicaise:1987}
also yields 
\[
\lambda_1\ge \frac14 h(\Gamma)^2
\]
if $\calV_0 \ne \emptyset$.  

As for upper bounds, a simple test function argument yields 
\begin{equation}\label{lmax.bound}
\lambda_1 \le \frac{\pi^2}{\ell_{\rm max}^2},
\end{equation}
where $\ell_{max}$ is the maximum edge length.  Similarly,
Berkolaiko-Kennedy-Kurasov-Mugnolo \cite[Thm.~1.3]{BKKM:2022} proved an upper bound 
\begin{equation}\label{girth.est}
\lambda_1 \le \frac{\pi^2}{\girth(\Gamma)^2},
\end{equation}
where the \emph{girth} is defined as the minimum cycle length of the graph formed from 
$\Gamma$ by identifying all Dirichlet vertices.  For higher eigenvalues, the same authors also proved
upper bounds in terms of the Betti number $\beta = E - V+1$, which counts the number of
independent cycles in $\Gamma$. In our notation, this yields \cite[Thm.~4.9]{BKKM:2017}
\begin{equation}\label{betti.est}
\lambda_k \le \paren[\Big]{k-\tfrac12 + \tfrac32 E - V_\std - \tfrac12 V_0}^{\!2} \frac{\pi^2}{L^2}.
\end{equation}

With this context established, we turn now to the question of estimates on $\lambda_2 - \lambda_1$ when
$\Gamma$ has at least one Dirichlet vertex.  As far as we are aware, there are no lower bounds in the literature
for this case. The combination of \eqref{lambda1.lower} and \eqref{girth.est} gives no bound on the gap, 
because the girth is obviously bounded by the total length.

By adapting the argument of Cheng and Oden \cite[Prop.~1.3]{ChengOden:1997}, we establish in \S\ref{wcheeger.sec}
a lower bound for $\lambda_2 - \lambda_1$ in terms of a weighted Cheeger constant, with weight given by the first eigenfunction $\phi_1$.
To make use of this bound, we prove a Harnack inequality for $\phi_1$ in \S\ref{phi1.sec}, with a constant depending only 
$L$ and $\ell_0$, the minimum edge length of $\Gamma$. 

The Harnack inequality (Proposition~\ref{phi1.min.prop}) 
implies an envelope estimate for $\phi_1$, which is perhaps of independent interest. Define 
the function 
\begin{equation}\label{upsilon.def}
\Upsilon(q) := \begin{cases} 1, & \dist(q, \calV_0) \ge \ell_0/2, \\
\sin(\tfrac{\pi}{\ell_0} \dist(q, \calV_0)), & \dist(q, \calV_0) < \ell_0/2. \end{cases}
\end{equation}
\begin{theorem}\label{envelope.thm}
Let $\Gamma$ be a compact metric graph with total length $L$ and minimum edge length $\ell_0$.
Suppose $\calV_0 \ne \emptyset$ and the first eigenfunction is normalized so that $\phi_1 \ge 0$ and $\norm{\phi_1} = 1$. 
There exists a constant $c_1(L,\ell_0)>0$ such that 
\[
c_1(L,\ell_0)\Upsilon \le \phi_1 \le \sqrt{\frac{2}{\ell_0}}\Upsilon.
\]
\end{theorem}

The combination of the weighted Cheeger inequality and the estimates of $\phi_1$ 
lead to the following result, whose proof is given in \S\ref{lowerbd.sec}:
\begin{theorem}\label{wcheeger.thm}
Let $\Gamma$ be a metric graph with $\calV_0 \ne \emptyset$. 
There exists a constant $C(L,\ell_0)>0$, depending only on the 
total length $L$ and minimum edge length $\ell_0$, such that
\begin{equation}\label{cll.lower}
\lambda_2 - \lambda_1 \ge C(L,\ell_0).
\end{equation}
\end{theorem}

For fully standard boundary conditions, $\lambda_1=0$ and the corresponding gap estimate is \eqref{std.lower}, depending on $L$ alone.
The following cases show that dependence on both $L$ and $\ell_0$ is required for graphs with at least one Dirichlet vertex.

\begin{figure}
\begin{overpic}[scale=.5]{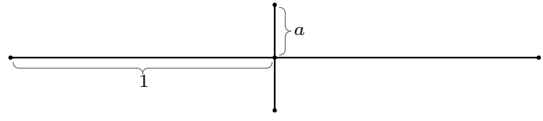}
\put(24.5,5){$\scriptstyle 1$}
\put(53.5,14.7){$\scriptstyle a$}
\end{overpic}
\caption{A $4$-star graph as in Example~\ref{star4.ex}.}\label{star4.fig}
\end{figure}

\begin{figure}
\begin{overpic}[scale=.5]{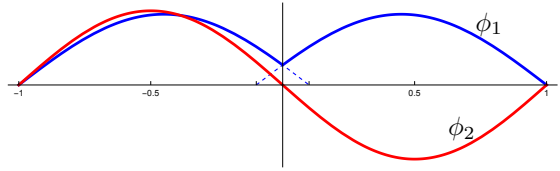}
\put(85,25){$\phi_1$}
\put(80,6){$\phi_2$}
\end{overpic}
\caption{A comparison of the first two eigenfunctions for the $4$-star graph of Example~\ref{star4.ex}.}\label{star4_phi12.fig}
\end{figure}

\begin{example}\label{star4.ex}
Let $\Gamma$ be a star graph with four edges, two of length $1$ and two of length $a <1$, as shown in 
Figure~\ref{star4.fig}. Dirichlet conditions are imposed at the four external vertices.  
Note that the total length $L=2(1+a) < 4$, $\ell_0 = a$, and the diameter is a constant $D=2$.

Define a coordinate $x$ so that 
the longer edges are parametrized by $x \in [0,1]$ and the shorter edges by $x \in [1,1+a]$.
In terms of this coordinate, the first eigenfunction can be written as 
\begin{equation}\label{phi1.star1a}
\phi_1(x) = \sin\paren*{\frac{\pi x}{1+a}},
\end{equation}
yielding
\[
\lambda_1 = \frac{\pi^2}{(1+a)^2}.
\]
The second eigenfunction vanishes on the short edges and is proportional to $\sin(\pi x)$ on the long edges, so that
$\lambda_2 = \pi^2$.  The spectral gap is thus
\[
\lambda_2 - \lambda_1 = \pi^2 \paren*{1- \frac{1}{(1+a)^2}},
\]
which is $\sim 2\pi^2a$ as $a \to 0$. 
\end{example}

\begin{example}\label{star15.ex}
Suppose $\Gamma$ is a star graph with one edge of length 2 and $k$ edges of length $1$, with Dirichlet vertex conditions
on all external vertices, as shown in Figure~\ref{star15.fig}.  Here $L = 2 + k$, $\ell_0 = 1$, and $D = 3$.

\begin{figure}
\begin{overpic}[scale=.5]{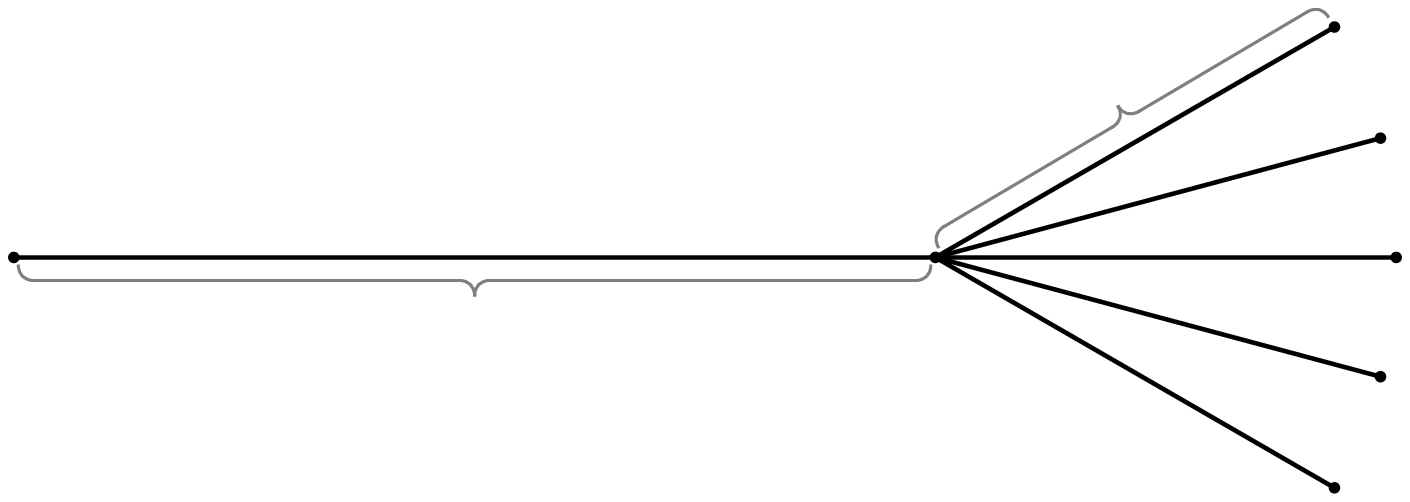}
\put(32.7,11.5){$\scriptstyle 2$}
\put(78,28){$\scriptstyle 1$}
\end{overpic}
\caption{A star graph as in Example~\ref{star15.ex}.}\label{star15.fig}
\end{figure}

\begin{figure}
\begin{overpic}[scale=.5]{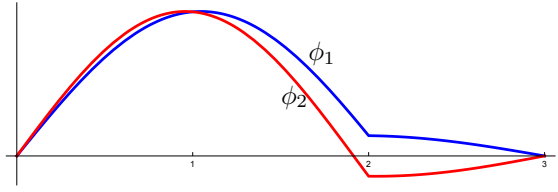}
\put(55,23){$\phi_1$}
\put(50,15){$\phi_2$}
\end{overpic}
\caption{Comparison of the first two eigenfunctions from Example~\ref{star15.ex}.}\label{star15_phi12.fig}
\end{figure}

For the two lowest eigenvalues, it suffices to consider eigenfunctions which do not vanish at the central vertex, 
and thus must take the same values on each smaller edge.  
We can thus use the linear coordinate $x \in [0,3]$, with the interior vertex located at  $x=2$.
Suppose an eigenfunction is given by 
\begin{equation}\label{star15.phi}
\phi(x) = \begin{cases} \sin (\sigma x), & x \in [0,2], \\
c \sin \sigma(3-x), & x \in [2,3].  \end{cases}
\end{equation}
The continuity and vertex conditions give $\sin (2\sigma) = c \sin \sigma$ and $\cos(2\sigma) = kc \cos \sigma$.
With double-angle formulas, this reduces to $(\tan \sigma)^2 = 2k+1$, yielding 
\[
\sigma \in \arctan \sqrt{2k+1} + \pi \bbZ.
\]
On the other hand, eigenfunctions which vanish at the central vertex have eigenvalues in $(\pi \bbN)^2$.
The first two eigenvalues are thus given by
\[
\lambda_1 = \paren*{\arctan \sqrt{2k+1}}^2, \quad \lambda_2 = \paren*{\pi - \arctan \sqrt{2k+1}}^2,
\]
with values just below and above $\pi^2/4$. 
For large $k$, we have $\lambda_2 - \lambda_1 \sim 2\pi/\sqrt{k}$.  In this example, the quantities $\ell_0 = 1$ and $D = 3$ are fixed,
while $L = k+2$.
\end{example}

Now let us turn to the issue of upper bounds on the gaps or ratios between consecutive eigenvalues.
For a general graph $\Gamma$ with $\calV_0 \ne \emptyset$, we have at least a bound deduced
from \eqref{lambda1.lower} and \eqref{betti.est}:
\begin{equation}\label{betti.ratio}
\frac{\lambda_{n+1}}{\lambda_n} \le (2n+1 + 3E - 2V_N - V_0)^2,
\end{equation}
provided $\Gamma$ is not a cycle graph.
The ratio $\lambda_2/\lambda_1$ can indeed be unbounded in general, as is 
demonstrated, for instance, in this paraphrase of \cite[Example~1.2]{DH:2010}:
\begin{example}\label{balloon.ex}
Consider an equilateral ``balloon'' graph, consisting of a pumpkin graph with $k$ edges of length 1, with a pendant 
edge attached at one vertex, as illustrated in Figure~\ref{balloon.fig}.  
Dirichlet conditions are imposed at the endpoint of the free edge.
The lowest two eigenvalues are equal on all edges of the pumpkin, so we can use the linear parameter $x \in [0,2]$,
with the range $[0,1]$ corresponding to the free edge.  For an eigenfunction of the form
\[
\phi(x) = \begin{cases} \sin (\sigma x), & x \in [0,1], \\
c \cos \sigma(2-x), & x \in [1,2],  \end{cases}
\]
the continuity and vertex conditions give $(\tan \sigma)^2 =1/k$.  The lowest two solutions give
\[
\lambda_1 = \arctan(1/\sqrt{k})^2, \qquad \lambda_2 = (\pi - \arctan(1/\sqrt{k}))^2,
\]
and hence
\[
\frac{\lambda_2}{\lambda_1} = \paren*{\frac{\pi}{\arctan(1/\sqrt{k})} - 1}^{\!2}.
\]
This gives $\lambda_2/\lambda_1 = 25$ for $k=3$, and $\lambda_2/\lambda_1 \sim \pi^2 k \text{ as }k \to \infty$.
\end{example}

\begin{figure}
\includegraphics[scale=.5,grid]{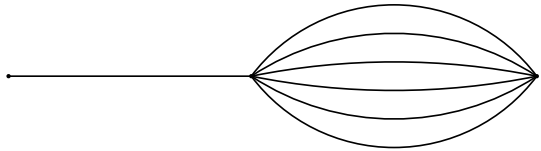}
\caption{An equilateral balloon graph with $k$ = 6, for which $\lambda_2/\lambda_1 \approx 50$.}\label{balloon.fig}
\end{figure}

\begin{figure}
\begin{overpic}[scale=.5]{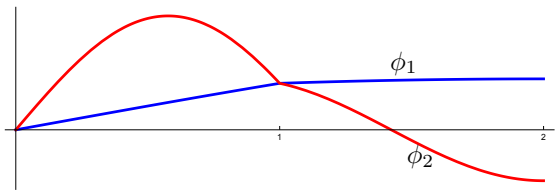}
\put(70,22){$\phi_1$}
\put(73,5){$\phi_2$}
\end{overpic}
\caption{Comparison of the first two eigenfunctions from Example~\ref{balloon.ex}.}\label{balloon_phi12.fig}
\end{figure}

In addition to showing that $\lambda_2/\lambda_1$ can be arbitrarily large, Example~\ref{balloon.ex} shows
the need for the cycle restriction in \eqref{betti.ratio}, as the balloon is a cycle graph.
The right side of \eqref{betti.ratio} reduces to $(3k+1)^2$ for this case, and so
the inequality fails as the ratio approaches $\pi^2$ for large $k$

For a metric tree $\Gamma$ with Dirichlet vertices on the external vertices, 
Nicaise \cite[Thm.~4.3]{Nicaise:1987} proved 
\begin{equation}\label{lratio.nicaise} 
\frac{\lambda_2}{\lambda_1} \le 2 + \sqrt{5},
\end{equation}
along with a bound analogous to Payne-P\'olya-Weinberger \cite{PPW:1956}: for $n \ge 2$,
\begin{equation}\label{nicaise.ppw}
\lambda_{n+1} - \lambda_n \le \frac{4}{n} \sum_{j=1}^n \lambda_j.
\end{equation}
The bound \eqref{nicaise.ppw} is derived using the approach that Hile-Protter \cite{HP:1980} developed for 
Dirichlet eigenvalues of domains in $\bbR^n$.
As in that source, the bound \eqref{nicaise.ppw} is the corollary of a more complicated but sharper estimate.
Again, for a metric tree $\Gamma$ with Dirichlet vertices on the external vertices,
\begin{equation}\label{nicaise.hp}
\lambda_{n+1} \le \sigma,
\end{equation}
where $\sigma$ denotes the unique root in $[\lambda_n, \infty)$ of the equation 
\begin{equation}\label{hp.sigma}
\sum_{j=1}^n \frac{\lambda_j}{\sigma - \lambda_j} = \frac{n}4.
\end{equation}
The bound \eqref{nicaise.ppw} is deduced from \eqref{nicaise.hp} by using $\lambda_{n+1} - \lambda_n \le \sigma - \lambda_j$ in \eqref{hp.sigma}.

In \S\ref{tree.sec}, we will show that the estimation techniques of Harrell-Stubbe \cite{HS:1997}, based on commutator formulas,
can be applied to metric trees by combining ideas from Nicaise \cite{Nicaise:1987} and Demirel-Harrell \cite{DH:2010}.  
This gives a more general bound which includes \eqref{nicaise.hp} as a special case. 

In the final section, \S\ref{extension.sec}, we consider extensions of the upper bounds on the ratio $\lambda_2/\lambda_1$
to metric graphs which are not trees.  In particular, we will see that the bound $\lambda_2/\lambda_1 \le 5$ continues to hold
for graphs consisting of trees to which pendant graphs are attached, provided the pendants are sufficiently small and include 
at least one Dirichlet vertex.

\section{Weighted Cheeger constant}\label{wcheeger.sec}

As in the introduction, $\Gamma$ denotes a compact, connected metric graph, and the Laplacian $-\Delta$
is defined with vertex conditions according to the decomposition $\calV = \calV_\std \cup \calV_0$.
The eigenvalues $\set{\lambda_j}$ are written in ascending order, starting from $\lambda_1 \ge 0$.
We may assume that the corresponding eigenfunctions $\phi_j$ are real and that 
$\phi_1 > 0$ away from $\calV_0$.

In this section we will establish a lower bound for $\lambda_2 - \lambda_1$ by adapting the weighted 
Cheeger constant from Cheng and Oden \cite{ChengOden:1997} 
to the metric graph setting. Given a continuous function $\phi:\Gamma \to \bbR$, we define
\begin{equation}\label{wcheeger.def}
h_\phi(\Gamma) := \inf \frac{\sum_S \phi^2}{\min(\int_{Y_1} \phi^2,\int_{Y_2} \phi^2)},
\end{equation}
where $S \subset \Gamma$ is a finite subset such that $\Gamma\backslash S$ is a disjoint union 
of non-empty open sets $Y_1$ and $Y_2$.  This decomposition, called a Cheeger cut, is illustrated in Figure~\ref{chcut.fig}.
Note that \eqref{wcheeger.def} agrees with \eqref{cheeger.def} when $\phi$ is constant.

\begin{figure}
\begin{overpic}[scale=.6]{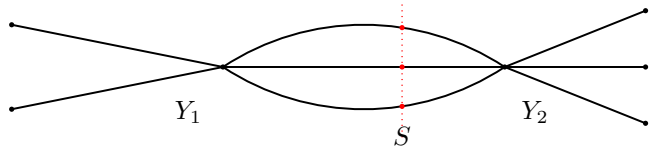}
\put(26,3){$Y_1$}
\put(60,-1){$S$}
\put(80,3){$Y_2$}
\end{overpic}
\caption{A Cheeger cut dividing the graph into two components.}\label{chcut.fig}
\end{figure}

\begin{theorem}\label{hphi.thm}
In the setting described above,
\[
\lambda_2 - \lambda_1 \ge \frac14 h_{\phi_1}(\Gamma)^2.
\]
\end{theorem}

Note that this result reduces to \eqref{std.cheeger} in the case $\calV_0 = \emptyset$.
The key step in the proof is the following estimate.
\begin{lemma}\label{fphi.lemma}
Let $f$ be a piecewise $C^1$ function on $\Gamma$, and for $\phi:\Gamma \to \bbR$ continuous, 
suppose that
\begin{equation}\label{fperp}
\int_\Gamma f \phi^2 = 0.
\end{equation}
Then 
\[
\norm{f'\phi} \ge \frac12 h_{\phi}(\Gamma)\> \norm{f\phi}.
\]
\end{lemma}
\begin{proof}
Set
\[
k := \sup \set*{t: \int_{\set{f\le t}} \phi^2 \le \tfrac12 \int_\Gamma \phi^2},
\]
and define
\[
g_+ = (f-k)^2 \chi_{\set{f >k}}.
\]
The level set $\set{g_+ = t}$ is finite for almost every $t > 0$, and the co-area formula gives
\begin{equation}\label{g.coarea}
\int_\Gamma \abs{g_+'} \phi^2 = \int_0^\infty \paren[\bigg]{\sum_{\set{g_+ = t}} \phi^2} dt.
\end{equation}
Note also that, by the choice of $k$,
\[
\int_{\set{g_+ > t}} \phi^2 \le \int_{\set{g_+ < t}} \phi^2
\]
for all $t >0$. Thus, for $t$ such that $\set{g_+ = t}$ is finite,
\[
\sum_{\set{g_+ = t}} \phi^2 \ge h_\phi(\Gamma) \int_{\set{g_+ > t}} \phi^2.
\]
Plugging this back into \eqref{g.coarea} gives
\begin{equation}\label{gplus.bd}
\begin{split}
\int_\Gamma \abs{g_+'} \phi^2 &\ge h_\phi(\Gamma) \int_0^\infty  \paren[\bigg]{ \int_{\set{g_+ > t}} \phi^2} dt \\
&= h_\phi(\Gamma) \int_\Gamma g_+ \phi^2.
\end{split}
\end{equation}

Similarly, for $g_- := (f-k)^2 \chi_{\set{f < k}}$ we obtain
\begin{equation}\label{gminus.bd}
\int_\Gamma \abs{g_-'} \phi^2 \ge  h_\phi(\Gamma) \int_\Gamma g_- \phi^2.
\end{equation}
Since $(f-k)^2 = g_+ + g_-$, adding \eqref{gplus.bd} and \eqref{gminus.bd} gives
\begin{equation}\label{gpm}
\int_\Gamma \paren*{\abs{g_+'} + \abs{g_-'}}\phi^2 \ge  h_\phi(\Gamma) \int_\Gamma (f-k)^2 \phi^2.
\end{equation}
After computing
\[
\abs{g_+'} + \abs{g_-'} = 2\abs{f-k}\,\abs{f'},
\]
we can apply Cauchy-Schwarz to \eqref{gpm} to obtain 
\[
\begin{split}
h_\phi(\Gamma) \>\norm{(f-k) \phi}^2 &\le 2\int_\Gamma \abs{f-k}\,\abs{f'}\phi^2 \\
&\le 2\> \norm{(f-k) \phi} \>\norm{f'\phi}.
\end{split}
\]
Hence 
\[
\norm{f'\phi} \ge \frac12 h_\phi(\Gamma) \>\norm{(f-k) \phi}.
\]
The final step is to note that the hypothesis \eqref{fperp} implies that
\[
\norm{(f-k) \phi}^2 = \norm{f\phi}^2 + k^2 \norm{\phi}^2 \ge \norm{f\phi}^2.
\]
\end{proof}

To deduce the weighted Cheeger bound from  Lemma~\ref{fphi.lemma}, we
call upon a variational principle from the folklore. In our notation,
\begin{equation}\label{folklore}
\lambda_2 - \lambda_1 = \inf \set*{\frac{\norm{f'\phi_1}^2}{\norm{f\phi_1}^2}:\> \int_\Gamma f \phi_1^2 = 0},
\end{equation}
for $f$ a piecewise $C^1$ function on $\Gamma$, with the minimum attained at $f = \phi_2/\phi_1$.
This formula is quite general and has been proven independently in a variety of contexts.
The earliest reference that we are aware of is Thompson-Kac \cite[eq.~(3.10)]{TK:1969}.
The short proof for metric graphs is included in the argument below.

\begin{proof}[Proof of Theorem~\ref{hphi.thm}]
Let $f = \phi_2/\phi_1$, which is smooth away from the vertices and satisfies
\[
\int_\Gamma f \phi_1^2 = \brak{\phi_2,\phi_1}  = 0.
\]
Assuming the eigenfunctions are normalized,
\begin{equation}\label{l2.fphi}
\begin{split}
\lambda_2 &= \norm{\phi_2'}^2 \\
&= \norm*{f'\phi_1 + f\phi_1'}^2 \\
&= \norm{f'\phi_1}^2 + \norm{f\phi_1'}^2 + \frac12 \int_\Gamma (f^2)' (\phi_1^2)'.
\end{split}
\end{equation}
Since both $f^2$ and $\phi_1^2$ satisfy the vertex conditions, we can integrate by parts to
obtain
\[
\begin{split}
\frac12 \int_\Gamma (f^2)' (\phi_1^2)' = - \frac12 \int_\Gamma f^2 \Delta(\phi_1^2) \\
= \int_\Gamma \paren*{\lambda_1 \phi_2^2 - f^2 (\phi_1')^2} \\
= \lambda_1 - \norm{f\phi_1'}^2.
\end{split}
\]
This simplifies \eqref{l2.fphi} to 
\begin{equation}\label{l21.fpp}
\lambda_2 - \lambda_1 = \norm{f'\phi_1}^2,
\end{equation}
which is the minimum case of \eqref{folklore}.
The result then follows from Lemma~\ref{fphi.lemma}, since
$f\phi_1 = \phi_2$.
\end{proof}

\section{Estimates of the first eigenfunction}\label{phi1.sec}

To make use of Theorem~\ref{hphi.thm} in the case $\calV_0 \ne \emptyset$, 
we need some control over the range of the first eigenfunction $\phi_1$.
In particular, we will establish a lower bound on a subset that excludes the Dirichlet vertices.

For each vertex $v_j \in \calV_0$, parametrize the edge incident to $v_j$ by $x_j \in [0,\ell_j]$, with $x_j = 0$
at $v_j$.  Within this edge define the interval,
\begin{equation}\label{Ij.def}
I_j := \sup \set*{x_j \in [0,\ell_j): \phi_1'(x_j) > 0},
\end{equation}
which includes the full interior of the edge unless $\phi_1$ has a local maximum.  Then let
\begin{equation}\label{gamma1.def}
\Gamma_1 := \Gamma \backslash (\cup_{v_j \in \calV_0} I_j).
\end{equation}
If $\phi_1$ has no local maxima in external edges, then $\Gamma_1$ is the subgraph obtained by 
trimming from $\Gamma$ all edges incident on $\calV_0$, as shown in Figure~\ref{gamma1.fig}.
If $\phi_1$ does have local maxima within an outer edge, then we would need to first add artificial 
vertices at these maxima before trimming.

\begin{figure}
\begin{overpic}[scale=.6]{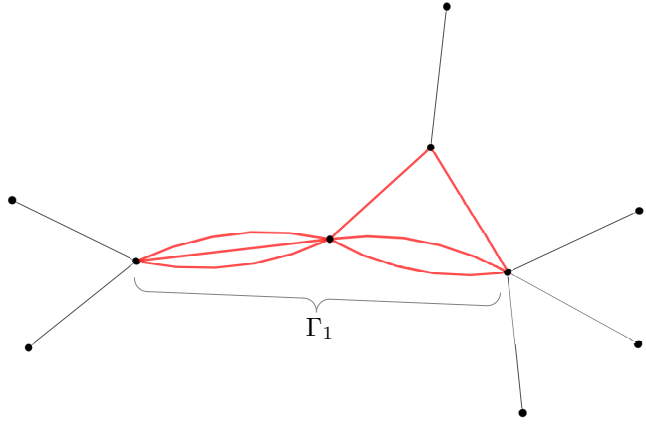}
\put(47,13){$\Gamma_1$}
\end{overpic}
\caption{The subgraph $\Gamma_1$ created by trimming external edges.}\label{gamma1.fig}
\end{figure}

Our goal in this section is to prove an explicit Harnack inequality for the restriction of 
$\phi_1$ to $\Gamma_1$. A general Harnack inequality for eigenfunctions on quantum graphs was
proven in Harrell-Maltsev \cite[Thm.~2.2]{HM:2020}.
The difference here is that we will produce a constant that depends only on $L$ and $\ell_0$.

By construction, the maximum value of the first eigenfunction,
\begin{equation}\label{m1.def}
M_1 := \max_{\Gamma} \phi_1,
\end{equation}
occurs at a point in $\Gamma_1$, and the minimum over $\Gamma_1$,
\begin{equation}\label{m1.def}
m_1 := \min_{\Gamma_1} \phi_1,
\end{equation}
is strictly positive. 
Note that it is possible for $\Gamma_1$ to consist of a single point. 
In this case, $\Gamma$ has $k = \#\calV_0$ edges, each 
connecting $v$ to a Dirichlet vertex. The continuity condition at $v$ implies that all of
these edges have equal length. 
Hence $\Gamma$ is a half-Dirichlet interval if $k=1$, a full Dirichlet interval for $k=2$,
and an equilateral star graph for $k \ge 3$.  

If $\Gamma_1$ contains more than one point, then $m_1$
occurs at a vertex $v \in \calV_\std$, by the concavity of $\phi_1$ on edges. 
The standard vertex condition implies that the outgoing derivative of $\phi_1$ at $v$ is $\le 0$ 
on at least one incident edge $e$ in $\Gamma$. 
By concavity, $\phi_1$ is strictly decreasing in the interior of $e$, and so $e$ must lie outside $\Gamma_1$. 
Therefore, the minimum $m_1$ is achieved at a vertex in $\calV_\std$ which is adjacent to a vertex in 
$\calV_0$. 

The fact that $\phi_1$ is strictly decreasing on an edge incident on a Dirchlet vertex implies that 
$\sigma \ell \le \pi/2$, where $\ell$ is the length of this edge and $\sigma := \sqrt{\lambda_1}$. 
This gives an upper bound
\begin{equation}\label{lam1.ell0.bd}
\lambda_1 \le \frac{\pi^2}{4\ell_0^2},
\end{equation}
provided $\Gamma_1$ contains more than one point. For the exceptional cases where $\Gamma_1$ contains a single point,
we see explicitly that $\lambda_1 = \pi^2/\ell_0^2$ for the full Dirichlet interval and $\lambda_1 = \pi^2/4\ell_0^2$ for all other cases.
Hence the Dirichlet interval is the only exception to the bound \eqref{lam1.ell0.bd}.

The estimate \eqref{lam1.ell0.bd} is implied by the girth estimate \eqref{girth.est} if $\calV_0$ contains more than one vertex, or if $\Gamma$ 
contains a cycle with at least two edges.  But it does cover a few additional cases, such as a tree or tadpole graph 
with a single Dirichlet vertex.

\begin{proposition}\label{phi1.min.prop}
Assuming that $\Gamma_1$ contains more than one point,
there exists a constant $c(L,\ell_0)>0$, depending only on $L = \abs{\Gamma}$ and $\ell_0$ the minimum edge length,
such that
\begin{equation}\label{phi.min.bd}
\frac{m_1}{M_1} \ge c(L,\ell_0).
\end{equation}
\end{proposition}

\begin{proof}
Let $\lambda_1= \sigma^2$ for $\sigma>0$. On each edge $e_j$ of $\Gamma_1$, $\phi_1$ is given by a 
positive arc of the sine function. Thus, we can choose phases $\alpha_j,\beta_j \in (0,\pi)$
such that
\begin{equation}\label{phi1.ej}
\phi_1|_{e_j}(x) = A_j \sin (\sigma x),
\end{equation}
for a parametrization of $e_j$ by 
\[
x \in \sqbrak*{\frac{\alpha_j}{\sigma},\frac{(\pi-\beta_j)}{\sigma}}.
\]
These phases are illustrated in Figure~\ref{phases.fig}.
By switching the orientation if necessary, we can assume that $\alpha_j \le \beta_j$. 

\begin{figure}
\begin{overpic}[scale=.5]{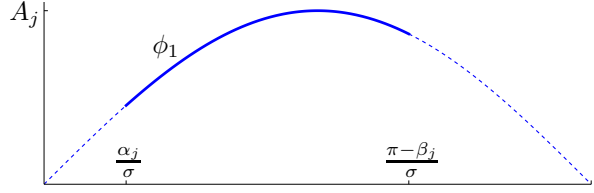}
\put(-6,30){$A_j$}
\put(13,3){$\tfrac{\alpha_j}{\sigma}$}
\put(62,3){$\tfrac{\pi - \beta_j}{\sigma}$}
\put(20,24){$\phi_1$}
\end{overpic}
\caption{The parametrization $\eqref{phi1.ej}$ with phases $\alpha_j$, $\beta_j$.}\label{phases.fig}
\end{figure}

The change in $\log \phi_1$ across $e_j$ is given by
\[
\log \phi_1\big|_{\alpha_j/\sigma}^{(\pi-\beta_j)/\sigma} = \log \frac{\sin \beta_j}{\sin\alpha_j}.
\]
If we define
\[
\delta_0 := \min_{e_j \subset \Gamma_1} \alpha_j,
\]
then for each edge $e_j$ in $\Gamma_j$,
\begin{equation}\label{delta0.est}
\abs*{\log \phi_1\big|_{\alpha_j/\sigma}^{(\pi-\beta_j)/\sigma}} \le - \log \sin \delta_0.
\end{equation}
Note that $\delta_0 < \pi/2$, since $\alpha_j + \beta_j< \pi$ on each edge.

Suppose that the vertices of $\Gamma_1$ where the minimum and maximum
of $\phi_1$ occur can be joined by a path with at most $q$ edges. The estimate \eqref{delta0.est} then
gives
\begin{equation}\label{log.minmax}
\log \frac{M_1}{m_1} \le - q \log \sin \delta_0,
\end{equation}
and the problem is now reduced to finding a lower bound for $\delta_0$.

Suppose that the minimal phase $\delta_0$ occurs at a vertex $v_0$ of $\Gamma_1$ which is 
the $x=0$ endpoint of an edge $e_j$ parametrized by $x \in [0,\ell_j]$.  In this parametrization,
\[
\phi_1(x) = A \sin (\sigma x + \delta_0).
\]
Because $\delta_0 < \pi/2$, $\phi_1$ does not have a local maximum at $v_0$, which implies that
$v_0 \in \calV_\std$.

The outward derivative into $e_j$ from $v_0$ is given by
\[
\phi_1'(0^+) = A \cos \delta_0.
\]
Let $e_k$ be the edge of $\Gamma$ incident to $v_0$ for which the inward-pointing derivative at $v_0$ is maximal. 
We can parametrize $e_k$ by $x \in [-\ell_k,0]$, and then continue the eigenfunction as
\begin{equation}\label{phi.B}
\phi_1(x) = B \sin (\alpha_k - \sigma x),\quad\text{for }x \le 0,
\end{equation}
for some phase $\alpha_k \in (0,\pi)$.  By continuity, the amplitudes satisfy
\begin{equation}\label{AB.cont}
A \sin \delta_0 = B \sin \alpha_k.
\end{equation}
Because the vertex condition at $v_0$ is standard, and by the choice of $e_k$,
we can estimate
\[
\phi_1'(0^+) \le (d_0-1) \phi_1'(0^-),
\]
where $d_0$ is the degree of $v_0$.  This gives
\[
A \cos \delta_0 \le (d_0-1) B\cos \alpha_k.
\]
Combining this with \eqref{AB.cont} yields a lower bound
\begin{equation}\label{tan.d0}
\tan \delta_0 \ge \frac{1}{d-1} \tan \alpha_k.
\end{equation}

To obtain a lower bound on the phase $\alpha_k$ from \eqref{phi.B}, note that since $\phi_1$ cannot vanish in 
the interior of $e_k$, $\alpha_k \ge \sigma\ell_k$. From \eqref{tan.d0}, we thus obtain
\begin{equation}\label{b.def}
\tan \delta_0 \ge \frac{1}{d_0 - 1} \tan\paren*{\sigma \ell_k}.
\end{equation}
If we denote the right side of \eqref{b.def} by $b$, then this gives
\[
\sin \delta_0 \ge \frac{b}{\sqrt{1+b^2}},
\]
and \eqref{log.minmax} implies that
\begin{equation}\label{min.phi.ratio}
m_1 \ge \paren*{\frac{b}{\sqrt{1+b^2}}}^{\!q} M_1.
\end{equation}
To complete the proof, note that $\ell_k \ge \ell_0$, both $q$ and $d_0$ are bounded by $L/\ell_0$ 
and $\sigma$ was bounded below by $\pi/2L$ in \eqref{lambda1.lower}.
\end{proof}

\begin{figure}
\begin{overpic}[scale=.5]{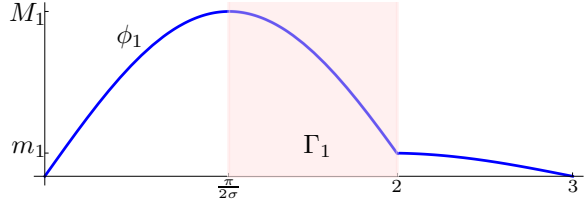}
\put(-4,6){$m_1$}
\put(-4.5,30){$M_1$}
\put(33,-1){$\scriptstyle \frac{\pi}{2\sigma}$}
\put(65,-1){$\scriptstyle 2$}
\put(97,-1){$\scriptstyle 3$}
\put(49,6){$\Gamma_1$}
\put(15,26){$\phi_1$}
\end{overpic}
\caption{Bounds on the first eigenfunction for a star graph as in Example~\ref{star15.ex}, with $k=100$.}\label{phi1plot.fig}
\end{figure}

Example~\ref{star15.ex} gives an interesting test case for the sharpness of the eigenfunction estimate.
The first eigenfunction is given by \eqref{star15.phi} with $\sigma = \arctan \sqrt{2k+1}$.
In this normalization $\max \phi_1 = 1$, which occurs at $x = \pi/2\sigma$.
The subgraph $\Gamma_1$ consists of the segment $[\pi/2\sigma, 2]$, and the minimum of 
$\phi_1$ on this segment occurs at $x=2$.  Thus
\[
\begin{split}
\frac{m_1}{M_1} &= \sin (2\sigma) \\
&= \frac{\sqrt{2k+1}}{k}.
\end{split}
\]
For this graph, the right side of \eqref{b.def} is
\[
\frac{1}{k} \tan \sigma = \frac{\sqrt{2k+1}}{k},
\]
and $q=1$. Hence, the lower bound \eqref{min.phi.ratio} is 
asymptotically sharp as $k \to \infty$ for this case.

We can extend the bound \eqref{phi.min.bd} beyond $\Gamma_1$.  On each interval $I_j = [0,\ell_j)$ from \eqref{Ij.def}, the
eigenfunction takes the form $\phi_1(x_j) = A_j \sin (\sigma x_j)$.  The points $x_j = \ell_j$ lie in $\Gamma_1$, so 
\[
A_j \sin (\sigma \ell_j) \ge m_1.
\]
Therefore, 
\begin{equation}\label{outer.bound}
\phi_1(x_j) \ge \frac{m_1}{\sin (\sigma \ell_j)} \sin (\sigma x_j)
\end{equation}
on each interval $I_j$. 

To prove Theorem~\ref{envelope.thm}, we first note that in the notation used above, 
the variable $x_j$ used to parametrize $I_j$ is equal to $\dist(\cdot,\calV_0)$ for $x_j \le \ell_0/2$.
The envelope function $\Upsilon$ defined in \eqref{upsilon.def} is equal to $1$ on $\Gamma_1$ and 
on the outer edges satisfies 
\begin{equation}\label{ups.ij}
\Upsilon|_{I_j}(x_j) = \begin{cases} 1, & x_j \ge \ell_0/2, \\
\sin(\tfrac{\pi}{\ell_0} x_j), & x_j < \ell_0/2. \end{cases}
\end{equation}
\begin{proof}[Proof of Theorem~\ref{envelope.thm}]
We can assume that $\Gamma_1$ consists of more than one point, since $\phi_1$ is easily computed 
explicitly in the exceptional cases.

From \eqref{outer.bound}, \eqref{ups.ij}, and the fact that $\sigma > \pi/2L$, we can immediately deduce the inequality,
\begin{equation}\label{phi1.ups.lower}
\phi_1 \ge m_1 \sin(\pi \ell_0/4L) \Upsilon,
\end{equation}
on all of $\Gamma$. On the other hand, since $\phi_1$ has the form
$A_j \sin (\sigma x_j)$ on the outer edges and $\sigma \le \pi/\ell_0$, it follows that
\begin{equation}\label{phi1.ups.upper}
\phi_1 \le M_1 \Upsilon.
\end{equation}

The constant $M_1$ satisfies the trivial inequality
\begin{equation}\label{M1.lower}
M_1 \ge \sqrt{\frac{2}{L}}.
\end{equation}
To complete the argument, we need to estimate $M_1$ from above. 
Suppose that the maximum value of $\phi_1$ is achieved at a point 
$q_1 \in \Gamma_1$. By the concavity of $\phi_1$ and the
vertex conditions, there exists a segment of $\Gamma$, parametrized by $y \in [0,\pi/2\sigma]$ with $y=0$ at $q_1$,
on which
\[
\phi_1(y) \ge M_1\cos(\sigma x).
\]
Integrating $\phi_1^2$ over this segment gives the inequality
\[
\pi M_1^2/4\sigma \le 1.
\]
Since $\sigma \le \pi/2\ell_0$ by \eqref{lam1.ell0.bd}, this proves
\begin{equation}\label{M1.upper}
M_1 \le \sqrt{\frac{2}{\ell_0}},
\end{equation}

Using the inequalities \eqref{phi1.ups.lower}, \eqref{phi1.ups.upper}, \eqref{M1.lower}, and
\eqref{M1.upper}, Theorem~\ref{envelope.thm} now follows from Proposition~\ref{phi1.min.prop}.
\end{proof}

\section{Estimation of the weighted Cheeger constant}\label{lowerbd.sec}

To complete the proof of Theorem~\ref{wcheeger.thm}, we will use Proposition~\ref{phi1.min.prop}
to estimate $h_{\phi_1}(\Gamma)$.  We will assume $\calV_0 \ne \emptyset$ throughout this discussion, 
since the bound \eqref{std.lower} already already covers the case of standard vertex conditions.

As we noted at the beginning of \S\ref{phi1.sec}, if $\Gamma_1$ consists of a single point $v$, 
then $\Gamma$ has $k$ equilateral edges connecting $v$ to Dirichlet vertices.
If $a$ denotes the edge length, then in all cases
\begin{equation}\label{gamma1.pt.eigv}
\lambda_1 = \frac{\pi^2}{4a^2},\quad \lambda_2 = \frac{\pi^2}{a^2}.
\end{equation}
Since $\ell_0 = 2a$ for $k=2$ and $\ell_0 = a$ for $k \ne 2$, the conclusion of 
Theorem~\ref{wcheeger.thm} holds trivially for these cases.

In the general case, our goal is to estimate $h_{\phi_1}(\Gamma)$ in terms of the unweighted 
Cheeger constant $h(\Gamma)$.  Since the latter satisfies the trivial bound $h(\Gamma) \ge 2/L$, 
this will complete the proof.

If $\Gamma$ is not an interval (and $\calV_0 \ne \emptyset$), 
then its Cheeger constant also satisfies a trivial upper bound,
\begin{equation}\label{h.ell0}
h(\Gamma) \le \frac{1}{\ell_0}.
\end{equation}
To see this, take a cut $S$ given by a single point on an outer edge of $\Gamma$, such that $Y_1$
is a segment of length $\ell_0$.

We can make a similar estimate for $h_{\phi_1}(\Gamma)$, in terms of the function
\begin{equation}\label{f.def}
f(x) := \frac{\sin^2 (\sigma x)}{\int_0^x \sin^2 (\sigma t) dt} 
= \frac{2\sin^2 (\sigma x)}{x - \frac{1}{2\sigma} \sin (2\sigma x)},
\end{equation}
defined for $x \in (0,\pi/\sigma)$.  It is easy to check that $f$ is decreasing on this interval.
\begin{lemma}\label{hphi.fell0.lemma}
If $\Gamma_1$ contains more than one point, then
\begin{equation}\label{hphi.fell0}
h_{\phi_1}(\Gamma) \le f(\ell_0),
\end{equation}
\end{lemma}
\begin{proof}
Taking a Cheeger cut $S$ consisting of a single point $x_j = \ell_0$ in one 
of the intervals $I_j$ from \eqref{Ij.def} gives the ratio
\begin{equation}\label{cheeger.ell0}
\frac{\sum_S \phi_1^2}{\int_{Y_1} \phi_1^2} = f(\ell_0),
\end{equation}
where $Y_1 = \set{0\le x_j < \ell_0}$ and $Y_2$ is the other component of $\Gamma\backslash S$. 
If $\calV_0$ contains more than one point, then we can choose $j$ so as to minimize $\int_{Y_1} \phi^2$ among
all $I_j$, guaranteeing that 
\begin{equation}\label{phi.y1y2}
\int_{Y_1} \phi_1^2 \le \int_{Y_2} \phi_1^2.
\end{equation}
On the other hand, if $\#\calV_0 = 1$, then, assuming that $\Gamma_1$ is not a single point, 
$\Gamma$ contains at least one interior edge
on which $\phi_1 \ge m_1$, implying that \eqref{phi.y1y2} holds also in this case.  From 
\eqref{cheeger.ell0} we thus obtain \eqref{hphi.fell0}.
\end{proof}

\begin{proposition}\label{hphi.prop}
If $\Gamma_1$ contains more than one point, then
\[
h_{\phi_1}(\Gamma) \ge \paren*{\frac{m_1 \sin(\sigma \ell_0/2)}{M_1}}^{\!2} h(\Gamma),
\]
where $\sigma = \sqrt{\lambda_1}$ and $M_1, m_1$ are the upper and lower bounds on $\phi_1$ from \S\ref{phi1.sec}.
\end{proposition}
\begin{proof}
By Lemma~\ref{hphi.fell0.lemma}, to estimate $h_{\phi_1}(\Gamma)$
we may limit our attention to Cheeger cuts $S$ for which
\begin{equation}\label{ratio.fell0}
\frac{\sum_S \phi_1^2}{\min(\int_{Y_1} \phi_1^2,\int_{Y_2} \phi_1^2)} \le f(\ell_0).
\end{equation}
By relabeling if necessary, we can also assume that the minimum in the denominator is the $Y_1$ integral, i.e.,
\begin{equation}\label{int.y1y2}
\int_{Y_1} \phi_1^2 \le \int_{Y_2} \phi_1^2.
\end{equation}
Let
\[
W := \set[\big]{q \in \Gamma:  \dist(q,\calV_0) \ge \ell_0/2}.
\]

\emph{Case 1:} Suppose that $S \cap W = \emptyset$. In the notation of \eqref{Ij.def},
\[
\Gamma\backslash W = \cup_j \set[\big]{0 \le x_j \le \ell_0/2}.
\]
Because $\phi_1$ is increasing as a function of $x_j$,
the convention \eqref{int.y1y2}  implies that $Y_1 \subset \Gamma\backslash W$ and 
$Y_2 \supset W$. Let $\calJ$ be the set of indices $j$ for which $S$ intersects $I_j$,
and set
\[
s_j = \max S \cap I_j.
\]
If the restriction of $\phi_1$ to $I_j$ is written as $A_j \sin (\sigma x_j)$, then
\[
\begin{split}
\sum_S \phi_1^2  &= \sum_{j\in \calJ} A_j^2 \sin^2 (\sigma s_j) \\
&= \sum_{j\in \calJ} f(s_j)  \int_0^{s_j} A_j^2 \sin^2 (\sigma x)\>dx.
\end{split}
\]
Since $f$ is decreasing and $s_j \le \ell_0/2$, this implies
\begin{equation}\label{ratio.fell2}
\sum_S \phi_1^2 \ge f(\ell_0/2) \int_{Y_1} \phi_1^2,
\end{equation}
which contradicts the assumption \eqref{ratio.fell0}.

\emph{Case 2:} Suppose that $S \cap W \ne\emptyset$ and $Y_1$ contains an interval 
$J \subset \Gamma\backslash W$. 
If $S'$ denotes the corresponding cut with the endpoints of $J$ deleted, then $Y_1$ is reduced 
to a component $Y_1' = Y_1\backslash J$, while $Y_2' = Y_2\cup J$.  The inequality \eqref{int.y1y2}
is still satisfied after the replacement, and 
\begin{equation}\label{ratio.J}
\sum_{\del J} \phi_1^2 \ge f(\ell_0/2) \int_{J} \phi_1^2,
\end{equation}
by the same argument used for \eqref{ratio.fell2}.
By the assumption \eqref{ratio.fell0}, and the fact that $f$ is strictly decreasing, 
we obtain 
\[
\frac{\sum_{S} \phi_1^2}{\int_{Y_1} \phi_1^2} < \frac{\sum_{\del J} \phi_1^2}{\int_{J} \phi_1^2}
\]
A comparison of cross-ratios then shows that
\[
\frac{\sum_{S} \phi_1^2}{\int_{Y_1} \phi_1^2} > \frac{\sum_{S'} \phi_1^2}{\int_{Y_1'} \phi_1^2}.
\]
That is, cutting the interval $J$ from $Y_1$ will reduce the Cheeger ratio.

\emph{Case 3:} Suppose that $S \cap W \ne\emptyset$ 
and $Y_1$ contains no interval in $\Gamma\backslash W$.
We continue to assume that the components $Y_j$ satisfy \eqref{ratio.fell0} and \eqref{int.y1y2}, which 
together imply that
\begin{equation}\label{ratio.y1}
\frac{\sum_S \phi_1^2}{\int_{Y_1} \phi_1^2} \le f(\ell_0).
\end{equation}
For a cut $S$ with these properties, the points of $S\backslash W$ (if any) correspond, 
in the parametrization \eqref{Ij.def}, to 
$x_j = b_j < \ell_0/2$, such that $ [0,b_j) \subset Y_2$. If we set 
\[
Z := \bigcap_{b_j \in S\backslash W} [0,b_j],
\]
then the set $S'' := S \cap W$ separates $\Gamma$ into components $Y_1'' := Y_1 \cup Z$ 
and $Y_2'' = Y_2\backslash Z$.

By \eqref{outer.bound}, $\phi_1$ satisfies a lower bound
\[
\min_W \phi_1 \ge m_1 \sin (\sigma\ell_0/2).
\]
We thus have
\begin{equation}\label{y1h.ineq}
\begin{split}
\frac{\sum_S \phi_1^2}{\int_{Y_1} \phi_1^2} 
&\ge \frac{\sum_{S''} \phi_1^2}{\int_{Y_1''} \phi_1^2} \\
&\ge \paren*{\frac{m_1 \sin (\sigma\ell_0/2)}{M_1}}^{\!2} \frac{ \#S''}{\abs{Y_1''}}.
\end{split}
\end{equation}

For $Y_2$ we can write the ratio as
\begin{equation}\label{sy2.pp}
\frac{\sum_S \phi_1^2}{\int_{Y_2} \phi_1^2} = 
\frac{\sum_{S''} \phi_1^2 + \sum_{S\backslash W} \phi_1^2}{\int_{Y_2''} \phi_1^2 + \int_{Z} \phi_1^2}.
\end{equation}
By \eqref{ratio.y1}, and the fact
\[
\sum_{S\backslash W} \phi_1^2 \ge f(\ell_0/2) \int_{Z} \phi_1^2,
\]
the decomposition \eqref{sy2.pp} shows that
\[
\frac{\sum_{S''} \phi_1^2}{\int_{Y_2''} \phi_1^2} \le \frac{\sum_S \phi_1^2}{\int_{Y_1} \phi_1^2}.
\]
The estimates of $\phi_1$ then yield
\begin{equation}\label{y2h.ineq}
\frac{\sum_S \phi_1^2}{\int_{Y_1} \phi_1^2} 
\ge \paren*{\frac{m_1 \sin (\sigma\ell_0/2)}{M_1}}^{\!2} \frac{ \#S''}{\abs{Y_2''}} .
\end{equation}
Combining \eqref{y1h.ineq} and \eqref{y2h.ineq} gives
\begin{equation}\label{case3.concl}
\frac{\sum_S \phi_1^2}{\int_{Y_1} \phi_1^2} 
\ge \paren*{\frac{m_1 \sin (\sigma\ell_0/2)}{M_1}}^{\!2} h(\Gamma).
\end{equation}

To summarize, Case 1 is ruled out by \eqref{ratio.fell0} and Case 2 can be reduced to
Case 3 with a reduction in the weighted Cheeger ratio. Hence the bound \eqref{case3.concl}
applies to $h_{\phi_1}(\Gamma)$.
\end{proof}

We can now complete the proof of Theorem~\ref{wcheeger.thm}. As noted above, the result follows 
from \eqref{gamma1.pt.eigv} if $\Gamma_1$ consists of a single point. Otherwise,
Propositions~\ref{phi1.min.prop} and \ref{hphi.prop} together show that
\[
\begin{split}
h_{\phi_1}(\Gamma) &\ge \paren[\Big]{c(L,\ell_0) \sin(\tfrac{\sigma \ell_0}2)}^{\!2} h(\Gamma) \\
&\ge \paren[\Big]{c(L,\ell_0) \sin(\tfrac{\pi \ell_0}{4L})}^{\!2} \frac{2}{L}.
\end{split}
\]
By Theorem~\ref{hphi.thm}, this proves \eqref{cll.lower} with
\[
C(L,\ell_0) = \frac{1}{L^2}  \paren[\Big]{c(L,\ell_0) \sin(\tfrac{\pi \ell_0}{4L})}^{\!4}.
\]

\section{Upper bounds for trees}\label{tree.sec}

The upper bounds for Dirichlet domains in $\bbR^n$ mentioned in the introduction are proven by creating a 
family of test functions from the products of eigenfunctions with an affine function on $\bbR^n$.  
The results of Nicaise \cite[Thm.~4.1]{Nicaise:1987} and Demirel-Harrell \cite[Thm.~2.5]{DH:2010}
use a similar construction, where an \emph{affine function} on a metric graph is interpreted as a continuous function which
is linear on each edge.  Throughout this section, we assume that 
$\Gamma$ is a metric tree, with $\calV_0$ consisting of the external vertices, meaning of degree one.

Our goal in this section is to adapt the techniques of Harrell-Stubbe \cite{HS:1997} to produce a
general eigenvalue estimate which generalizes the Nicaise bound \eqref{nicaise.hp}. Our first step
is to produce a family of affine functions whose derivatives cover $\Gamma$ uniformly in an average sense.

Let $\calA(\Gamma)$ denote the space of affine functions on $\Gamma$ which satisfy
standard vertex conditions at the points of $\calV_\std$. No vertex condition are imposed at the points of $\calV_0$.
If we interpret $\Gamma$ as an electric circuit, with each edge assigned a resistance equal to its length, then
functions $\calA(\Gamma)$ corresponds precisely to a voltage function satisfying the Ohm's and 
Kirchhoff's circuit laws. As Kirchhoff \cite{Kirchhoff:1847} demonstrated in 1847, there exists a voltage function
for any combination of external voltages applied at the points in $\calV_0$. Hence, $\calA(\Gamma)$
contains non-constant functions provided $\calV_0$ contains at least two points.

From Nicaise \cite[Lemma~4.2]{Nicaise:1987} we quote the following result. The original did not include a proof, 
so we will give one here. A similar result was derived independently in Demirel-Harrell \cite[Thm.~2.9]{DH:2010}, 
but without the restriction to three functions.
\begin{lemma}\label{Acover.lemma}
For a metric tree $\Gamma$, there exist functions $g_\alpha \in \calA(\Gamma)$ for $\alpha \in \set{1,2,3}$
such that on all edges, $\abs{g_\alpha'}$ equals $0$ or $1$ and
\begin{equation}\label{galpha.sum}
\sum_{\alpha=1}^3 \abs{g_\alpha'} = 2.
\end{equation}
\end{lemma}
\begin{proof}
Let us refer to a subgraph of $\Gamma$ consisting of a vertex with two adjoining external edges (leaves) as a \emph{leaf-pair}. 
Trimming a leaf pair from a vertex of degree $\ge 3$ reduces its degree by 2. 
If all possible leaf-pairs are trimmed from a given vertex, the result either an artificial (degree 2) or 
an external vertex (degree 1).  By carrying out this trimming process as far as possible at each vertex, we eventually 
reduce $\Gamma$ to a single segment. Hence $\Gamma$ can be constructed by starting from a single segment and
attaching leaf-pairs successively. Each leaf-pair is added by gluing its vertex to any point on the graph, which could be an existing
vertex or an edge point.

The family $\set{g_\alpha}$ is constructed by induction, using this decomposition of $\Gamma$.
For the initial segment, we may choose an arbitrary parametrization $x$ and 
set $g_1(x) = x$, $g_2(x) = -x$, and $g_3(x) = 0$.

Now suppose the family $\set{g_\alpha}$ has been defined with the desired properties for a tree $\Gamma$. 
Let $\tilde{\Gamma}$ be a graph obtained
by adding a single leaf-pair to $\Gamma$. The extensions $\tilde{g}_\alpha$ may be defined as follows:
\begin{enumerate}
\item Suppose the leaf-pair is attached at an internal vertex of $\Gamma$ (possibly artificial), so that each $g_\alpha$ already satisfies 
standard vertex conditions at this point. We can extend the family so that both $\tilde{g}_1'$ and $\tilde{g}_2'$ alternate 
$\pm 1$ on edges of the leaf-pair, while $\tilde{g}_3$ is constant on these edges.  
\item If the leaf-pair is attached an an external vertex of $\Gamma$, then by construction two of the $g_\alpha$, say $\alpha=1,2$ will 
have derivatives $\pm 1$ at this vertex and $g_3$ will be constant. We extend  the family  so that $\tilde{g}_1'$ is zero on one leaf of the pair,
$\tilde{g}_2'$ is zero on the other, and the derivatives on the other leaves are chosen to satisfy the vertex condition. The third function,
whose derivative vanishes into the vertex, is extended so that $\tilde{g}_3'$ alternates $\pm 1$ on the new leaves.
\end{enumerate}
At every stage of the induction, each function satisfies the vertex conditions and
$\abs{g_\alpha'} = 1$ for two values of $\alpha$ and $\abs{g_\alpha'} =0$ for the third.
\end{proof}

Using the collection $\set{g_\alpha}$, we can prove the main result of this section,
a version of the general eigenvalue inequality from Harrell-Stubbe \cite[Thm.~5]{HS:1997}.
\begin{theorem}\label{HS.thm}
Let $\Gamma$ be a metric tree with Dirichlet conditions on external vertices.
For $z \in (\lambda_n,\lambda_{n+1}]$, suppose that $f$ is a positive function on the set $\set{\lambda_1,\dots,\lambda_n}$ such that 
$f(\lambda_j)/(z-\lambda_j)^2$ is nondecreasing with $j$.  Then
\[
\sum_{j=1}^n f(\lambda_j) \le 4 \sum_{j=1}^n \frac{f(\lambda_j)}{z-\lambda_j} \lambda_j.
\]
\end{theorem}
\begin{proof}
Let $G_\alpha$ be the multiplication operator on $L^2(\Gamma)$ associated to $g_\alpha$.  We define a corresponding set of first-order
differential operators
\[
D_\alpha := \frac12 [\Delta, G_\alpha].
\]
On an edge parametrized by $x$, $D_\alpha$ acts as $(\del_x g_\alpha) \del_x$.  Since $g_\alpha'$ takes values in $\set{0,\pm 1}$,
the commutator $[D_\alpha, G_\alpha]$ is the projection onto the support of $g_\alpha'$.  By the construction in 
Lemma~\ref{Acover.lemma},
\[
\sum_{\alpha=1}^3 [D_\alpha, G_\alpha] = 2.
\]

Thus we can write
\[
\begin{split}
\sum_{j=1}^n f(\lambda_j) &= \tr \paren*{P_n f(-\Delta)} \\
&= \frac12 \sum_{\alpha=1}^3 \tr \paren[\big]{P_n f(-\Delta)[D_\alpha, G_\alpha]},
\end{split}
\]
where $P_n$ denotes the spectral projection onto the eigenspace for
$\set{\lambda_1,\dots,\lambda_n}$.
Expanding the trace, just as in the proof of \cite[Thm.~1]{HS:1997}, then gives
\[
\begin{split}
\sum_{j=1}^n f(\lambda_j)  &= - \sum_{\alpha=1}^3 \sum_{j=1}^n \sum_{\substack{m=1 \\ \lambda_m \ne \lambda_j}}^n 
\frac{f(\lambda_j) - f(\lambda_m)}{\lambda_j - \lambda_m} \abs{\brak{D_\alpha \phi_j, \phi_m}}^2 \\
&\qquad + 2 \sum_{\alpha=1}^3 \sum_{j=1}^n \sum_{q = n+1}^\infty \frac{f(\lambda_j)}{\lambda_q - \lambda_j} \abs{\brak{D_\alpha \phi_j, \phi_q}}^2.
\end{split}
\]
By the hypotheses on $f$, the argument from proof of \cite[Thm.~5]{HS:1997} then applies directly to give
\[
\sum_{j=1}^n f(\lambda_j) \le 2\sum_{\alpha=1}^3 \sum_{j=1}^n \frac{f(\lambda_j)}{z-\lambda_j} \norm{D_\alpha \phi_j}^2.
\]
By the construction of $g_\alpha$,
\[
\sum_{\alpha=1}^3 \norm{D_\alpha \phi_j}^2 = 2 \norm{\phi_j'}^2 = 2\lambda_j,
\]
which completes the proof.
\end{proof}

Taking $f=1$ in Theorem~\ref{HS.thm} yields the Hile-Protter bound \eqref{nicaise.hp} obtained by Nicaise.
Using $f(\lambda) = (z-\lambda)^2$ gives an inequality
\begin{equation}\label{DH.ineq}
\sum_{j=1}^n (z - \lambda_j)(z - 5 \lambda_j) \le 0,
\end{equation}
for $z \in [\lambda_n,\lambda_{n+1}]$, which was obtained previously by Demirel-Harrell \cite[Eq.~(3.15)]{DH:2010}.
The discriminant of the quadratic 
polynomial on the left side of \eqref{DH.ineq} is positive and the roots must lie outside the interval $(\lambda_n, \lambda_{n+1})$.
This yields the following result, analogous to \cite[Prop.~6]{HS:1997}:
\begin{theorem}\label{quadratic.gap.thm}
For $\Gamma$ a metric tree with Dirichlet vertices on the external vertices, the quantity
\[
D_n :=  \paren[\bigg]{\frac{3}{n} \sum_{j=1}^n \lambda_j}^{\!2} - \frac{5}{n} \sum_{j=1}^n \lambda_j^2
\]
satisfies $D_n \ge 0$ for all $n \ge 2$.  Furthermore, the eigenvalues satisfy the inequalities
\[
\lambda_{n} \ge \frac{3}{n} \sum_{j=1}^n \lambda_j - \sqrt{D_n},\qquad 
\lambda_{n+1} \le \frac{3}{n} \sum_{j=1}^n \lambda_j + \sqrt{D_n},
\] 
and hence
\[
\lambda_{n+1} - \lambda_{n} \le 2 \sqrt{D_n}.
\]
\end{theorem}

For $n=1$ we have $D_1 = 4\lambda_1^2$, so the estimate reduces to $\lambda_2/\lambda_1 \le 5$, 
equivalent to \eqref{nicaise.ppw} and \eqref{nicaise.hp} but weaker than \eqref{lratio.nicaise} .  
To compare the estimates for $\lambda_3$, let us define the moments for the first pair of eigenvalues,
\[
a_1 := \tfrac12 (\lambda_1 + \lambda_2), \qquad a_2 := \tfrac12 (\lambda_1^2 + \lambda_2^2).
\]
Then the estimate from \eqref{quadratic.gap.thm} reads
\begin{equation}\label{quad3.bd}
\lambda_3 \le 3a_1 + \sqrt{9a_1^2 - 5a_2}.
\end{equation}
On the other hand, the Hile-Protter type bound \eqref{nicaise.hp} reduces to
\[
\lambda_3 \le 3a_1 + \sqrt{9a_1^2 - 5\lambda_1\lambda_2}.
\]
Since $\lambda_1\lambda_2 \le a_2$, by the geometric mean inequality, the bound \eqref{quad3.bd}
is stronger.

\section{Extensions of the upper bound}\label{extension.sec}

In this section we investigate the possibility for extending the upper bounds discussed in \S\ref{tree.sec} to
graphs which are modifications of trees. We will consider two possibilities: adding edges between existing vertices
of the tree and attaching pendant graphs.

We cannot expect to produce a set of affine functions satisfying \eqref{galpha.sum} on a general graph. 
However, it is worth noting that the existence of a single function $h \in \calA(\Gamma)$ gives a bound 
in terms of $\phi_1$.  Returning to the electric circuit analogy from \S\ref{tree.sec}, let us define a (scalar)
\emph{current} as a function $\eta: \Gamma \to [0,\infty)$ which is constant on each edge and
which, under some choice of edge orientations, satisfies Kirchhoff's current law.  This is equivalent to the condition
\begin{equation}\label{eta.hprime}
\eta = \abs{h'}\text{ for some }h \in \calA(\Gamma).
\end{equation}

The following inequality is implicit in the proofs of the Payne-P\'olya-Weinberger estimate $\lambda_2/\lambda_1 \le 5$
in various cases.  For convenience we include a direct proof.
\begin{lemma}\label{l21.lemma}
If $\Gamma$ is a metric graph that admits a non-zero current function $\eta$, then
\begin{equation}\label{l21.estimate}
\lambda_2 - \lambda_1 \le 4\> \frac{\norm{\eta\phi_1'}^2}{\norm{\eta\phi_1}^2}.
\end{equation}
\end{lemma}
\begin{proof}
Given $\eta$, choose $h$ according to \eqref{eta.hprime}.  By shifting $h$ by a constant if necessary,
we can assume that $u := h\phi_1$ is orthogonal to $\phi_1$.  Then min-max gives the estimate
\begin{equation}\label{l2u.minmax}
\lambda_2 \le \frac{\norm{u'}^2}{\norm{u}^2}.
\end{equation}
Because $u$ satisfies the vertex conditions, by the assumptions on $h$, we can integrate by parts to compute
\[
\begin{split}
\norm{u'}^2 &= \brak{u,-\Delta u} \\
&= \brak{u,\lambda_1 u - 2h'\phi_1'} \\
&= \lambda_1 \norm{u}^2 - 2\brak{u, h'\phi_1'}.
\end{split}
\]
By \eqref{l2u.minmax} this gives
\begin{equation}\label{gap.A}
\lambda_2 - \lambda_1 \le \frac{A}{\norm{u}^2},
\end{equation}
where 
\[
A := - 2\brak{u, h'\phi_1'}.
\]
The Cauchy-Schwarz estimate gives
\begin{equation}\label{a2.cs}
A^2 \le 4 \norm{u}^2 \norm{\eta\phi_1'}^2.
\end{equation}
On the other hand, we can compute using integration by parts,
\[
\begin{split}
A
&= - \frac12 \int_\Gamma (h^2)' (\phi_1^2)' \\
&= \frac12 \int_\Gamma (h^2)'' \phi_1^2 \\
&= \norm{\eta\phi_1}^2.
\end{split}
\]
By \eqref{a2.cs} this yields
\[
\frac{A}{\norm{u}^2} \le \frac{\norm{\eta\phi_1'}^2}{\norm{\eta\phi_1}^2},
\]
and the result follows from \eqref{gap.A}.
\end{proof}

As noted in \S\ref{tree.sec}, any graph $\Gamma$ for which $\calV_0$ has at least two vertices
will admit non-zero current functions. However, to obtain universal bounds from \eqref{l21.estimate} we need either some uniformity in the
choice of $\eta$ or some control over the behavior of $\phi_1$.

\subsection{Saguaro graphs}
The first observation is that the conclusion of Lemma~\ref{Acover.lemma} continues to hold for 
graphs constructed from trees by adding edges uniformly. That is, suppose $\Gamma$ is constructed from a tree
by replacing each internal edge by a pumpkin with $k>1$ edges of the same length, and each external 
edge by a star with k edges of the same length. We will call the result a \emph{saguaro} graph, after the 
tree-like cactus. 

\begin{figure}
\includegraphics[scale=.5]{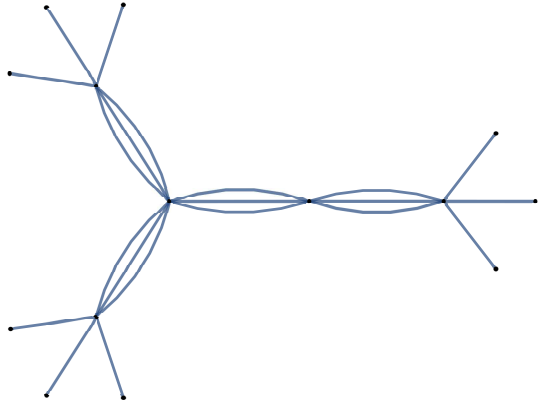}
\caption{A saguaro graph.}\label{saguaro.fig}
\end{figure}

Given a saguaro graph $\Gamma$ based on the tree $\Gamma_{\rm T}$, we can apply Lemma~\ref{Acover.lemma} 
to produce a trio of affine functions $\set{g_\alpha}$ on $\Gamma_{\rm T}$.
These functions can be extended to elements of $\calA(\Gamma)$  satisfying \eqref{galpha.sum} 
by simply replicating the values on parallel edges.
The existence of this tree implies that all of the universal tree bounds carry over to $\Gamma$, including 
\[
\frac{\lambda_2}{\lambda_1} \le 2 + \sqrt{5}, \qquad \lambda_{n+1} - \lambda_n \le \frac{4}{n} \sum_{j=1}^n \lambda_j,
\] 
from Nicaise \cite[\S4]{Nicaise:1987}, along with the general bounds from Theorem~\ref{HS.thm}.

We can also consider the case of a modified saguaro graph, where the 
number of edges of each pumpkin or star varies between values $k_{\rm min}$ and $k_{\rm max}$.
We continue to assume that the edges have equal lengths within each pumpkin or star.
Suppose $\set{g_\alpha}$ denotes the trio of affine functions produced by Lemma~\ref{Acover.lemma}
for the underlying tree $\Gamma_{\rm T}$. On $\Gamma$ we can produce a corresponding set of currents $\set{\eta_\alpha}$
by subdividing the current $g_\alpha'$ among the parallel edges in each segment.  
That is, on a segment (pumpkin or star) with $k$ edges, 
we set $\eta_\alpha := \abs{g'_\alpha}/k$. Because $\sum_\alpha  \abs{g'_\alpha}^2 = 1$ on 
$\Gamma_{\rm T}$, this construction gives 
\[
\frac{1}{k^2_{\rm max}} \le \sum_{\alpha=1}^3 \eta_\alpha^2 \le \frac{1}{k^2_{\rm min}}.
\]
By summing the inequality 
\[
(\lambda_2 - \lambda_1) \norm{\eta_\alpha\phi_1}^2  \le 4 \norm{\eta_\alpha\phi_1'}^2
\]
over $\alpha$, we obtain
\[
\frac{\lambda_2}{\lambda_1} \le 1 + 4 \frac{k^2_{\rm max}}{k^2_{\rm min}}.
\]

\subsection{Ornamented trees}

Suppose $\Gamma$ is constructed from a tree graph $\Gamma_{\rm T}$ by attaching pendant graphs $P_j$, $j=1,\dots m$,
to internal (and possibly artificial) vertices of $\Gamma_{\rm T}$.  Let $q_j \in \calV_\std$ denote the attachment vertex for $P_j$.
See Figure~\ref{pendant.fig} for an illustration of this \emph{ornamented tree} graph.

\begin{proposition}
Let $\Gamma$ be an ornamented tree consisting of metric tree $\Gamma_{\rm T}$, with Dirichlet conditions at exterior vertices,
and pendants $P_1, \dots, P_m$ each containing at least one Dirichlet vertex. 
Suppose that
\[
\lambda_1 \le \lambda_1(P_j)
\]
for each $j$, where $\set{\lambda_n}$ denotes the spectrum of $\Gamma$ and 
$\lambda_1(P_j)$ is the first eigenvalue of $P_j$ defined by assigning standard boundary conditions at 
the attachment point $q_j$.
Then
\[
\frac{\lambda_2}{\lambda_1} \le 5.
\]
In particular, this inequality holds if 
\begin{equation}\label{pj.lmax}
\abs{P_j} \le \frac12 \ell_{\rm max}(\Gamma_T)
\end{equation}
for each $j$.
\end{proposition}

\begin{figure}
\begin{overpic}[scale=.5]{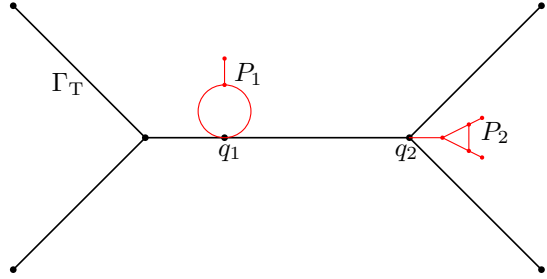}
\put(8,34){$\Gamma_{\rm T}$}
\put(39,22){$q_1$}
\put(42,36){$P_1$}
\put(72,22){$q_2$}
\put(88,25){$P_2$}
\end{overpic}
\caption{An ornamented tree with two pendants.}\label{pendant.fig}
\end{figure}

\begin{proof}
Let $\set{g_\alpha} \in \calA(\Gamma_{\rm T})$ be the trio of functions obtained by Lemma~\ref{Acover.lemma}.
We define corresponding current functions on $\Gamma$ by setting $\eta_\alpha = \abs{g_\alpha'}$ on edges
of $\Gamma_T$ and $\eta_\alpha = 0$ on each pendant $P_j$.  This gives
\[
\frac12 \sum_{\alpha=1}^3 \norm{\eta_\alpha \phi_1}^2 = \int_{\Gamma_{\rm T}} \phi_1^2,
\]
and
\[
\frac12 \sum_{\alpha=1}^3 \norm{\eta_\alpha \phi_1'}^2 = \int_{\Gamma_{\rm T}} (\phi_1')^2.
\]
Thus, by Lemma~\ref{l21.lemma},
\begin{equation}\label{gt.ratio}
\lambda_2 - \lambda_1 \le 4 \frac{\int_{\Gamma_{\rm T}} (\phi_1')^2}{\int_{\Gamma_{\rm T}} \phi_1^2}.
\end{equation}
An integration by parts gives
\begin{equation}\label{pj.parts}
\int_{\Gamma_{\rm T}} (\phi_1')^2 = \lambda_1 \int_{\Gamma_{\rm T}} \phi_1^2 + \sum_{j=1}^m \phi_1(q_j) \del_{P_j} \phi_1(q_j),
\end{equation}
where $\del_{P_j} \phi_1(q_j)$ denotes the sum of derivatives of $\phi_1$ at $q_j$ into the incident edges of $P_j$.

Now for each $j$ let $u_j$ denote the first eigenfunction of $P_j$, with standard vertex conditions imposed at $q_j$,
so that $u_j \ge 0$ on $P_j$ and 
\[
-\Delta u_j = \lambda_1(P_j) u_j.
\]
By Green's identity, and the fact that $\del_{P_j} u_j(q_j) = 0$,
\[
\int_{P_j} (- \phi_1 \Delta u_j + u_j \Delta \phi_1) = - u_j(q_j) \del_{P_j} \phi_1(q_j).
\]
On the other hand,
\[
\int_{P_j} (- \phi_1 \Delta u_j + u_j \Delta \phi_1) = (\lambda_1(P_j) - \lambda_1) \int_{P_j} \phi_1 u_j.
\]
Since the eigenfunctions are positive, we conclude that $\lambda_1(P_j) - \lambda_1 \ge 0$ implies that
\[
\del_{P_j} \phi_1(q_j) \le 0.
\]
Under this assumption, \eqref{pj.parts} gives
\[
\int_{\Gamma_{\rm T}} (\phi_1')^2 \le \lambda_1 \int_{\Gamma_{\rm T}} \phi_1^2,
\]
and it follows from \eqref{gt.ratio} that $\lambda_2/\lambda_1 \le 5$.

By the general bounds \eqref{std.lower} and \eqref{lmax.bound}, the eigenvalue condition will hold provided
\[
\abs{P_j} \le \frac12 \ell_{\rm max}(\Gamma).
\]
Clearly it suffices to compute $\ell_{\rm max}$ over $\Gamma_{\rm T}$ rather than $\Gamma$.
\end{proof}


\end{document}